\documentclass{article}
\usepackage{amsthm,amssymb}

\newtheorem{theorem}{Theorem}[section]
\newtheorem{lemma}[theorem]{Lemma}

\theoremstyle{definition}
\newtheorem{definition}[theorem]{Definition}
\newtheorem{example}[theorem]{Example}

\theoremstyle{remark}
\newtheorem{remark}[theorem]{Remark}

\newcommand{\A}{\mathbb{A}}
\newcommand{\p}{\mathbb{P}}
\newcommand{\Z}{\mathbb{Z}}
\newcommand{\Q}{\mathbb{Q}}
\newcommand{\R}{\mathbb{R}}
\newcommand{\Id}{\mathrm{Ideal}}
\newcommand{\Sing}{\mathrm{Sing}}
\newcommand{\ord}{\mathrm{ord}}
\newcommand{\Cl}{\mathrm{Closure}}
\newcommand{\Tr}{\mathrm{Transform}}

\newcommand{\Tf}{\mathrm{Tightify}}

\begin{document}

\title{A Simplified Game for Resolution of Singularities}

\author{Josef Schicho}




\maketitle

\begin{abstract}
We will describe a combinatorial game that models the problem of resolution of singularities
of algebraic varieties over a field of characteristic zero. By giving a winning strategy for
this game, we give another proof of the existence of resolution. 
\end{abstract}

\section*{Introduction}

The proof of existence and construction of resolution of singularities of algebraic varieties in characteristic zero
can be divided into two parts. First, there is an algebraic part, providing necessary
constructions such as blowups, differential closure,
transforms along blowup, descent in dimension, transversality conditions, and properties
of these constructions. Second, there is a combinatorial part that consists in the setup
of a tricky form of double induction taking various side conditions into account. The combinatorial
part can be formulated as a game. The two parts of the proof
can be cleanly separated: once the properties of the algebraic constructions
are clear, it is no more necessary to do any algebra in the induction proof. In \cite{Schicho:12d},
the algebraic parts and the combinatoric parts of the proof are in separate sections that
are logically independent of each other. The formulation is based on Villamayor's
constructive proof \cite{Villamayor:89,Villamayor_Encinas:97b,Villamayor:91}, using ideas from other proofs
\cite{Encinas_Hauser:02,Schicho:99f,Wlodarczyk:07,Kollar:07,Blanco_Encinas:11}. It is needless to say that there are many
more proofs that indirectly influenced our formulation. We just mention
\cite{Hironaka:64,Bierstone_Milman:97}. For a more exhaustive account of
other proofs, see \cite{Schicho:12d} and the references cited therein.

In this paper, we give a simplified version of the game described in \cite{Schicho:12d}.
In contrast to the description there, the combinatorial part is not entirely independent of
the algebraic part. In the game there, there are two players, one who tries to resolve
and one who provides combinatorial data for the singularities according to a set of rules 
(and who is destined to lose). In this paper, the second player is replaced by an algebraic oracle
that has complete information on the singularity, so that the rules are not needed. Also, the
combinatorial data has been reduced: the stratification of the singular locus is not used
any more. 

The game described here has been introduced at the Clay Summer School in Obergurgl, 2012 on
Resolution of Singularities. For me,
this event was a unique experience full of intensive interactions with many highly qualified
young researchers. Several simplifications are owed to the participants,
for instance the precise notion of the equivalence relation. 
The simplified game has actually been played at the summer school; see section \ref{sec:play}
for a description of this play. 

As a consequence of the winning strategy for the second player, we get another formulation
of the resolution algorithm. A new aspect of this formulation is that it does not depend on local
choices. All substeps in the algorithm can be done globally or locally, whichever is more useful.
There is a single substep which requires an intermediate passing to a local cover, but the result
of this substep is again global (see Remark~\ref{rem:global}).

Most proofs in Section~\ref{sec:alg} have been done as exercises in the Clay Summer School; we give
them here (mostly through references) just for the sake of completeness. Apart from these,
the existence proof of resolution in this paper is self-contained, and we hope that it serves as a gentle
introduction/explanation of this classical result.

This version of the paper has been read by four reviewers, and I would like to thank them for their truly formidable reviews. 
They contained in total 116 suggestions,
some of them independently by several reviewers, 9 additional references to the literature, on 17 pages in total. 
There was not a single comment which was not clear. I tried to follow most of their suggestions; many
remarks and examples, for instance Example~\ref{ex:moeb} of a singularity for which there is no global
descent, are only here because of their persistence.

\section{Habitats, Singularities, and Gallimaufries} \label{sec:alg}

In this section, we introduce the algebraic concepts which are needed for our setup of the
resolution problem and algorithm: habitats, singularities, transform along blowups, 
differential closure, gallimaufries and descents. The terminology used in this paper
is the same as in \cite{Schicho:12d}.

Let $K$ be a field of characteristic zero.

\begin{definition}
A {\em habitat} over $K$ is an equidimensional nonsingular algebraic variety $W$ over $K$
together with a finite sequence of nonsingular hypersurfaces $(E_1,\dots,E_r)$ such that
no two have a common component, and such that their sum is a normal crossing divisor.
We denote this habitat by $(W,(E_1,\dots,E_r))$, and when the hypersurface sequence is
not important, we denote the habitat by $(W,\ast)$.
\end{definition}

Examples of Habitats are the affine spaces $\A^n$, $n\ge 0$, with divisors defined by coordinates
$x_i$, $1\le i\le n$. In the analytic category, every habitat is locally isomorphic to such a habitat.
It is also possible that some of the hypersurfaces are empty. This is a necessity
because we consider habitats as local/global objects where we would like to restrict to open sets,
or glue together habitats on an open cover when the restrictions to the intersections coincide.
If we restrict to an open subset of the complement of $E_i$, $1\le i\le r$, then the restricted
habitat has an empty hypersurface at the $i$-th place.

\begin{definition}
A subvariety $Z$ of a habitat as above is called {\em straight} iff
it is pure-dimensional, and for every point $p\in Z$, there is a regular system of parameters $u$
such that $Z$ is locally defined by a subset of $u$ and every hypersurface of the habitat sequence
that contains $p$ is defined by an element of $u$. This concept also arises in \cite{Schicho:12d},
where it is called {\em transversal}, and in \cite{Villamayor:91}, where it is called {\em normal crossing}.
\end{definition}

For instance, in the habitat $(\A^n,(x_1,\dots,x_n))$, the variety defined by $(x_1$, $\dots$, $x_m)$, $m\le n$,
is straight. If the habitat sequence is empty, then straightness is equivalent to smoothness.

\begin{definition}
Let $Z$ be a straight subvariety of a habitat $(W,(E_1,\dots,E_n))$.
The {\em blowup} along $Z$ is the habitat $(W',(E_1',\dots,E_n',E_{n+1}))$, where $W'$ is the
blowup of $W$ along $Z$, $E_i'$ is the strict transform of $E_i$ for $i=1,\dots,n$,
and $E_{n+1}$ is the exceptional divisor introduced by the blowup (the inverse image of the
center $Z$).
\end{definition}

We allow the following degenerate cases: if $Z=E_1$, then the blowup is $(W,(\emptyset,E_2,\dots,E_n,E_1))$.
If $Z=W$, then the blowup is $(\emptyset,(\emptyset,\dots,\emptyset))$.

\begin{definition}
For any habitat $(W,\ast)$, we define an operator $\Delta$ from the set of ideal sheaves on $W$ to itself, 
as follows. 
For $I\subset \mathcal{O}_W$ and affine open subset $U\subset W$, $\Delta(I)|_U$ 
is the ideal sheaf generated 
by $I|_U$ and all first order partial derivatives of elements in $I|_U$.

The $i$-th iteration of the operator $\Delta$ is denoted by $\Delta^i$.
\end{definition}

\begin{definition}
A {\em singularity} on a habitat $(W,\ast)$ is a finitely generated sheaf of Rees algebras
$A=\oplus_{i=0}^\infty A_i$ over $A_0=\mathcal{O}_W$, i.e. a
sequence of ideal sheafs $A_i\subset\mathcal{O}_W$
such that $A_0=\mathcal{O}_W$ and $A_i\cdot A_j\subseteq A_{i+j}$ and equality holds for sufficiently
large indices $i,j$ (this is equivalent to finite generation).

We say a singularity $A$ is of {\em ideal-type} if there is an integer $b>0$ and ideal sheaf $I$ 
such that $A_{nb}=I^n$
for all indices which are multiples of $b$, and $A_i=(0)$ otherwise. These singularities
are denoted by $(I,b)$. (This is Hironaka's notion of pairs.)

The {\em singular locus} $\Sing(A)$ of a singularity $A=\oplus_{i=0}^\infty A_i$ is the
intersection of the zero sets of $\Delta^{i-1}(A_i), i>0$. We say that a singularity
is {\em resolved} if its singular locus is empty.
\end{definition}

\begin{remark}
The above concept of singularity is based on Hironaka's definition \cite{Hironaka:77} of
idealistic exponents (ideal-type singularities). I learned the description of singularities in terms
by Rees algebras from \cite{Encinas_Villamayor:03}, which is based on \cite{Villamayor:08}.
Similar description by algebras or filtration
of rings have been used systematically in \cite{Hironaka:05,Kawanoue:07,Kawanoue:10}. Note that
our definition of Rees algebras slightly differs from the definition in \cite{Swanson_Huneke:06}.
\end{remark}

Note that the intersection defining the singular locus is a finite intersection by N\"otherianity.
For computing the singular locus, it suffices to consider generating degrees.

The singular locus of an ideal-type singularity of the form $(I,1)$ is just the zero set of $I$.
The singular locus of an ideal-type singularity of the form $(I,b)$ with $b>1$ is the zero set of points
where the order of $I$ is at least $b$.

The trivial singularities are the zero singularity $A_i=0$ for $i>0$ and the unit singularity
$A_i=\mathcal{O}_W$ for all $i\ge 0$. The singular locus of the zero singularity is $W$, and
the singular locus of the unit singularity is the empty set (and so, the unit singularity is resolved).

\begin{definition}
Let $Z$ be a straight subvariety in the singular set of a singularity $A$. The {\em transform} of $A$
is the singularity $A'=\oplus_{i=0}^\infty A'_i$ on the blowup $(W,(¸\ast,E_{n+1}))$, where $A'_i$ 
is such that $f^\ast(A_i)=A_i\mathcal{O}_{W'}=\Id(E_{n+1})^i\cdot A'_i$ for $i>0$.
Recall that $E_{n+1}=f^{-1}(Z)$ is the exceptional divisor.
\end{definition}

\begin{example} \label{ex:blow}
We consider the ideal-type singularity $(\langle x^2-y^3\rangle,2)$ in the habitat $(\A^2,())$.
The singular locus of $A$ is the only point where $x^2-y^3$ has order~2, namely $(0,0)$.
The blowup of $\A^2$ can be covered by two open affine charts, with coordinates $(x,\tilde{y}=\frac{y}{x})$
and $(\tilde{x}=\frac{x}{y},y)$, respectively. In the first chart, the transform is the
ideal-type singularity $(\langle 1-x\tilde{y}^3\rangle,2)$; in the second chart, it is the ideal-type singularity
$(\langle\tilde{x}^2-y\rangle,2)$. 
Note that in both charts the singularity is resolved.
\end{example}

\begin{definition}
A {\em thread} is a sequence of singularity-habitat pairs, where the next
is the transform of the previous under blowup of a straight subvariety in the singular locus.
If the last singularity has empty singular locus, then we say the thread is a {\em resolution} of the
first singularity of the thread.
\end{definition}

For instance, the transform in the example above has empty singular locus. Therefore the
thread consisting of the single blowup above is a resolution of the singularity $(\langle x^2-y^3\rangle,2)$.

The objective in this paper is to show that every singularity admits a resolution.
Desingularization of algebraic varieties over characteristic zero is then a consequence.

\begin{theorem} \label{thm:reduce}
Assume that every singularity over $K$ has a resolution. Then every irreducible variety $X$ over $K$ that
can be embedded in a nonsingular ambient space has a
desingularization, i.e. a proper birational map from a nonsingular variety to $X$.
\end{theorem}

\begin{proof}
Let $X\subset W$ be a variety embedded in a nonsingular ambient space $W$.
If $X$ is a hypersurface, then we simply resolve the singularity $(\Id(X),2)$. The proper
transform of $X$ is then a subscheme of the transform of $(\Id(X),2)$. Since the transform
of $(\Id(X),2)$ has no points of order~1, also the proper transform has no points of order 1,
and therefore it is a nonsingular hypersurface.

In higher codimension, there exist singular varieties with an ideal of order~1,
namely varieties that are embedded in some smooth hypersurface;
so it is not enough to resolve the $(\Id(X),2)$. 
Instead, we resolve the singularity $(\Id(X),1)$ and take only the part of the resolution
where the proper transform of $X$ is not yet blown up. In the next step, when the proper transform
is blown up, it must be a nonsingular subvariety. So the thread defines a sequence of blowing ups
such that the proper transform is nonsingular, and this is a desingularization.
\end{proof}

\begin{remark}
The resolutions obtained in this way are {\em embedded resolutions}: the singular variety $X$ is
embedded in a nonsingular ambient space $W$, and one constructs a proper birational morphism $\pi:\tilde{W}\to W$
such that the proper transform of $X$ is nonsingular. Moreover, the morphism $\pi$ is an isomorphism at
the points outside $X$ and at the smooth points of $X$. In the hypersurface case, the singularity is
already resolved in a neighborhood of these points. In the general case, the singularity is resolved at
the points outside $X$, and it can be resolved by a single blowing up step locally in a neighborhood of
a smooth point of $X$. But this is precisely the step where the resolution of the singularity is truncated.
\end{remark}

\begin{definition}
Let $A_1$ and $A_2$ be two singularities on the same habitat. Then their sum $A_1+A_2$ is defined as the
singularity defined by the Rees algebra generated by $A_1$ and $A_2$.
\end{definition}

\begin{remark}
If $A_1=(I_1,b)$ and $A_2=(I_2,b)$ are ideal-type singularities with the same generating degree, then
$A_1+A_2=(I_1+I_2,b)$. For ideal-type singularities with different generating degrees, there is no
such easy construction.
\end{remark}

\begin{lemma}
The singular locus of $A_1+A_2$ is equal to $\Sing(A_1)\cap\Sing(A_2)$. If $Z$ is a straight subvariety
contained in this intersection, then $\Tr_Z(A_1+A_2)=\Tr_Z(A_1)+\Tr_Z(A_2)$.
\end{lemma}

\begin{proof}
Straightforward.
\end{proof}

As a consequence, resolution of $A+B$ separates the singular loci. More precisely, the resolution of $A+B$
defines threads starting with $A$ and $B$, and the singular sets of the final singularities of these
threads have disjoint singular loci.

\begin{definition}
A singularity $A=\oplus_{i=0}^\infty A_i$ is {\em differentially closed} iff $\Delta(A_{i+1})\subseteq A_i$
for all $i\ge 0$.

The {\em differential closure} of a singularity $A$ is the smallest differentially closed singularity
containing $A$.
\end{definition}

A priori it is not clear if the differential closure exists, or in other words if the intersection 
of all differentially closed finitely generated Rees algebras containg $A$ is again finitely generated. 
Assume that
$A$ has generators $f_i$ in degree $d_i$, $i=1,\dots,N$. Then one can use the Leibniz rule to
show that the differential closure is generated by all partial derivatives of order $j<d_i$ in degree $d_i-j$.

\begin{remark}
The notion of differential closure is closely related to differential Rees algebras used in \cite{Villamayor:08}
and with differential saturation of an idealistic filtration defined in \cite{Kawanoue:07}. These two
cases are different but both use higher order differential operators. Here we only use first order
differential operators; this would not work for positive characteristic.
\end{remark}

\begin{definition}
Two singularities $A,B$ on the same habitat are {\em equivalent} iff there exists $N>0$
such that $\Cl(A)_{kN}=\Cl(B)_{kN}$ for all $k\in\Z_+$.
\end{definition}

\begin{lemma}
If two singularities $A$ and $B$ are equivalent, then their singular loci coincide.

Assume that $A$ and $B$ are equivalent, and let $Z$ be a straight subvariety in the singular locus.
Then the transforms of $A$ and $B$ on the blowup at $Z$ are again equivalent.
\end{lemma}

\begin{proof}
The first statement is straightforward. For the second statement, we use \cite[Lemma~9]{Schicho:12d}:
If $C$ is the differential closure of $A$, and $A'$ and $C'$ are the transforms of $A$ and $C$ along
a center inside the singular locus, then the differential closures of $A'$ and $C'$ are equal.
(In general, the transform of a differentially closed singularity may not be differentially closed.)
\end{proof}

\begin{definition} \label{def:gendeg}
A number $b>0$ is a {\em generating degree} of a singularity $A$ iff $A$ is equivalent to the
ideal-type singularity $(A_b,b)$.
\end{definition}

If $A$ is generated by elements of $A_b$ as an algebra over $\mathcal{O}_W$, then it is
an easy exercise that $b$ is a generating degree for $A$. 

\begin{definition}
A {\em subhabitat} of a habitat $(W,(E_1,\dots,E_n))$ is a straight subvariety
$V\subset W$ which does not have components that are contained in one of the $E_i$, 
together with the sequence of the
intersections $(V\cap E_1,\dots,V\cap E_n)$.
\end{definition}

If $Z$ is a straight subvariety of a subhabitat $(V,\ast)$ of $(W,\ast)$, then it is also
a straight subvariety of $(W,\ast)$. The blowup of $(V,\ast)$ at $Z$ is a subhabitat of
the blowup of $(W,\ast)$ at $Z$; its underlying variety $V'$ is the proper transform of $V$.

\begin{example}
Let $0\le l\le m\le n$.
On the habitat $(\A^n,())$, we have the subhabitat $(V,())$ where $V$ is the hypersurface defined
by $x_{m+1}=\dots=x_n=0$ (say that $x_1,\dots,x_n$ are the coordinate variables). 
Note that $V$ is isomorphic to $\A^m$.
Let $Z$ be the subvariety defined by $x_{l+1}=\dots=x_n=0$. Then the blowup is covered by
$n-l$ charts with coordinate functions $x_1,\dots,x_l,x_k,\frac{x_{l+1}}{x_k},\dots,\frac{x_n}{x_k}$,
where $k=l+1,\dots,n$. The proper transform of $V$ has a non-empty intersection with the $m-l$
charts corresponding to $k=l+1,\dots,m$. It is isomorphic to to the blowup of $\A^m$ at the
subvariety defined by the last $m-l$ coordinates.
\end{example}

\begin{definition}
Let $i:V\to W$ be the inclusion map of a subhabitat $(V,\ast)$ of $(W,\ast)$.
The {\em restriction} of a singularity $B=\oplus_{i=0}^\infty B_i$ on $(W,\ast)$ to $(V,\ast)$ is
defined as the singularity $A=\oplus_{i=0}^\infty A_i$ where
$A_i:=i^\ast(B_i)\mathcal{O}_V$ and $i^\ast(B_i)$ denotes the pullback of $B_i$ along the
inclusion map.

If $(\Id(V),1)$ is a subalgebra of $B$, then we say that $B$ {\em restricts properly} to $V$. 
\end{definition}

\begin{example}
Let $A$ be the ideal-type singularity $(\langle x,y^2-z^3\rangle ,1)$ on the habitat $(\A^3,())$. Then
the hyperplane $x=0$ is a subhabitat to which $A$ restricts properly.
\end{example}

\begin{remark}
If $V$ has codimension~1, then the statement ``$B$ {\em restricts properly} to $V$'' is equivalent
to the statement ``$V$ is a hypersurface of maximal contact'' in \cite{Villamayor:91}.
\end{remark}

If $B$ restricts properly to $V$, then the singular locus of $B$ is contained in $V$ and
is equal to the singular locus of the restriction of $B$ to $V$. The proof is straightforward.

Assume that $B$ restricts properly to $V$, and let $A$ be the restriction. Let $Z\subset\Sing(B)\subset V$
be a straight subvariety. Then the transform of $A$ on the blowup $(V',\ast)$ is equal to
the restriction of the transform of $B$ to the restriction to the subhabitat $(V',\ast)$.

\begin{definition}
Let $i:V\to W$ be the inclusion map of a subhabitat $(V,\ast)$ of $(W,\ast)$.
Let $A$ be a differentially closed singularity on $(V,\ast)$. Then the {\em extension} of $A$ to $(W,\ast)$ 
is defined as the largest differentially closed algebra which is contained in $\oplus_{i=0}^\infty (i^\ast)^{-1}(A_i)$.
\end{definition}

\begin{example}
On the habitat $(\A^2,())$, we consider the singularity $(\langle x,y^2\rangle,1)+(\langle y^3\rangle,2)$
(this is the differential closure of $(\langle x^2+y^3\rangle,2)$). Its restriction to the subhabitat defined by $x$
is equal to $(\langle y^2\rangle,1)+(\langle y^3\rangle,2)$. The inverse of the pullback is
$(\langle x,y^2\rangle,1)+(\langle x,y^3\rangle,2)$. It is not differentially closed, because $\partial_x(x)=1$ is not
contained in the degree~1 component. When we remove $x$ from the list of generators in degree~2, we
get $(\langle x,y^2\rangle,1)+(\langle y^3\rangle,2)$, and this is the extension. 
\end{example}

\begin{lemma}
Let $i:V\to W$ be the inclusion map of a subhabitat $(V,\ast)$ of $(W,\ast)$.
Let $B$ be a differentially closed singularity on $(W,\ast)$ which restricts properly to $V$. Then
the extension of the restriction of $B$ to $V$ is equal to $B$.

Let $A$ be a differentially closed singularity on $(V,\ast)$. Then the extension of $A$ restricts properly to $V$,
and its restriction is equal to $A$.
\end{lemma}

\begin{proof}
This is \cite[Theorem~11]{Schicho:12d}. It compares to the ``commutativity'' statement 
in \cite{Encinas_Hauser:02}.
\end{proof}

\begin{definition}
Let $(W,\ast)$ be a habitat, and let $m\le\dim(W)$ be a non-negative integer.
A {\em gallimaufry} of dimension $m$ on $(W,\ast)$ is a differentially closed singularity $A$, such that for every 
point $p$ in the singular locus, there is an open subset $U\subset W$ and a subhabitat
$(V,\ast)$ of the open restriction $(U,\ast)$ of dimension $m$, such that $A|_U$ restricts properly to $V$.
Such an open subhabitat is called {\em zoom} for $A$ at $p$.

Let $A$ be a gallimaufry of dimension $m>0$.
Assume that there exists an open cover of the habitat of $A$ such that for every open subset $U$,
there is a subhabitat of dimension $m-1$ to which $A|_U$ properly restricts; in other words,
$A$ can also be considered as a gallimaufry of dimension $m-1$. 
Then we say that the gallimaufry $A$ descends to dimension $m-1$.
\end{definition}

Any singularity can be considered as a gallimaufry of dimension $\dim(W)$. Any gallimaufry
of dimension $m<\dim(W)$ can also be considered as a gallimaufry of dimension $m+1$. 
The dimension of the singular locus of a gallimaufry is less than or equal to the dimension 
of the gallimaufry.
The unit singularity on $(W,\ast)$ can be considered as a gallimaufry of any dimension $m\le\dim(W)$.

\begin{example} \label{ex:moeb}
Here is an example that shows that the passage to local covers is really necessary.
Let $C\subset\A^3$ be an affine smooth space curve which is a complete intersection with ideal
generated by $F,G\in K[x,y,z]$. Let $f:C\to\p^1$ be a regular map that cannot be extended
to $\A^3$. For instance, we could set $K=\R$ and $C$ to be the circle $x^2+y^2-1=z=0$ and $f$ as the map
$(x,y,z)\to (y:1-x)=(1+x:y)$. Let $I$ be the ideal of all functions $g$ vanishing along
$C$ such that for all $p\in C$, the gradient of $p$ is a multiple of 
the gradient of $f_1(p)F+f_2(p)G$,
where $f(p)=(f_1(p):f_2(p))$. In the concrete case of the unit circle and $f$ as above, the ideal
is generated by $(x^2+y^2)(1+x)+yz,(x^2+y^2-1)y+z(1-x),(x^2+y^2-1)z,z^2$. 

The ideal-type singularity $(I,1)$ is locally analytically isomorphic to $(\langle x,y\rangle, 1)$,
bacause locally analytically we can assume $C$ is the line $x=y=0$ and the gradient of elements
in the ideal are multiples of the gradient of $x$. Still locally analytically, the hyperplane defined by
$x$ is a subhabitat with proper restriction. In the concrete example of the unit circle, one can cover
$\A^3$ by the three open subsets: $U_1$ is defined by $x+1\ne 0$ and removing the point $(x,y,z)=\left(\frac{1}{3},0,0\right)$,
$U_2$ is defined by $x-1\ne 0$, and $U_3$ is the complement of the unit circle. In $U_1$, the restriction to the subhabitat
defined by $(x^2+y^2-1)(1+x)+yz$ is proper (the only singular point of the habitat has been removed from $U_1$); 
in $U_2$, the restriction to the subhabitat $(x^2+y^2-1)y+z(1-x)$ is proper;
in $U_3$, the singularity is resolved, so the restriction to any subhabitat is proper. Therefore we can
consider $(I,1)$ as a gallimaufry in dimension~2.

On the other hand, we can show that $(I,1)$ does not globally restrict properly to a subhabitat of dimension~2.
Assume, indirectly, that $H$ is the equation of such a surface. Then $H$ lies in the ideal of $C$, hence
we can write $H=AF+BG$ for some $A,B\in K[x,y,z]$. Then $(x,y,z)\mapsto (A(x,y,z):B(x,y,z))$ would be an extension
of $f:C\to\p^1$, contradicting our assumption that such an extension does not exist.
\end{example}


\begin{lemma}
Let $A$ be a gallimaufry of dimension $m>0$. 
If $A$ descends to $m-1$, then any transform of $A$ also descends to $m-1$.
\end{lemma}

\begin{proof}
If $V$ is a subhabitat of dimension $m-1$ to which $A$ properly restricts, then
the transform restricts properly to the proper transform of $V$.
\end{proof}

\begin{definition}
A gallimaufry $A$ of dimension $m$ is {\em bold} if $\dim(\Sing(A))=m$.
\end{definition}

\begin{example}
A gallimaufry $A$ of maximal dimension $n=\dim(W)$ is bold if and only if there is an irreducible
component of $W$ on which $A$ is the zero singularity.
\end{example}

\begin{lemma} \label{lem:bold}
Let $A$ be a bold singularity. Then the $m$-dimensional locus $Z$ of $\Sing(A)$ is straight,
and the transform of $A$ on the blowup along $Z$ is not bold.
\end{lemma}

\begin{proof}
Let $V$ be a subhabitat of dimension $m$ to which $A$ properly restricts.
Then $\Sing(A)$ is the equal to the union of all irreducible components $V_0$ of $V$ such that the restriction
to $V_0$ is the zero singularity. Since $V$ is straight, it follows that $Z$ is straight.

Let $\pi:W'\to W$ be the blowup along $Z$.
Locally at some neighbourhood of a point $p$ in an irreducible component $V_0$ of $Z$, 
the blowup manifold of the subhabitat $V$ is empty. Hence the blowup of $V$ along $Z$ just
removes all components in $Z$. It follows that the transform of the restricted singularity is
not bold. Hence the transform of the $A$ as a gallimaufry of dimension $m$ is not bold.
\end{proof}

The next definition introduces a numerical invariant of the order of the singularity. It is
based on Hironaka's order function. Other authors used iterated order functions to construct
an invariant governing the resolution process. Here, the resolution process should not be
determined by an invariant, but we still need some order concept.

\begin{definition}
Let $(I,b)$ be an ideal-type singularity on $(W,(E_1,\dots,E_r))$ which is not bold
as a gallimaufry in dimension $\dim(W)$, i.e. $I$ 
is not the zero ideal on any component of $W$. Let $\{i_1,\dots,i_k\}$ be the set of all 
hypersurface indices $i$ such that $E_i\cap\Sing(A)\ne\emptyset$. The {\em monomial factor}
of $(I,b)$ is defined as the sequence 
$\left(\frac{e_1}{b},\dots,\frac{e_k}{b}\right)$ such that
$I$ $\subseteq$ $\Id(E_{i_1})^{e_1}\cdots\Id(E_{i_k})^{e_k}$, with integers $e_1,\dots,e_k$ chosen
as large as possible.

The {\em maxorder} of $(I,b)$ is defined as
$\frac{\min\{a\mid \Delta^a(\tilde{I})=\langle 1\rangle\}}{b}$,
where $\tilde{I}$ is the ideal sheaf $\Id(E_1)^{-e_1}\cdots\Id(E_r)^{-e_r}I$.
Note that this the maximum of the function $p\mapsto \ord_p(\tilde{I})/b$.

For an arbitrary singularity that is not bold, monomial factor and maxorder are defined by passing
to an equivalent ideal-type singularity.

For a gallimaufry that is not bold, monomial factor and maxorder are defined by restricting to
a subhabitat of correct dimension.
\end{definition}

In order to show that the definitions are valid for arbitrary singularities, we use
that any singularity is equivalent to an ideal-type singularity, by the comment after 
Definition~\ref{def:gendeg}.
Moreover, one needs to show that two equivalent ideal-type
singularities have the same monomial factors and order; in our setup, this is a straightforward
consequence of the statement that if the ideal-type singularities $(I_1,b_1)$ and $(I_2,b_2)$
are equivalent, then there exist positive integers $n_1,n_2$ such that $b_1n_1=b_2n_2$
and $I_1^{n_1}=I_2^{n_2}$.
One may compare with \cite{Hironaka:77}, where te independence of the choice of ideal-type
representative is shown for a similar equivalence relation.
The validity of the definitions for gallimaufries (independence of the choice of the local subhabitat)
is a consequence of \cite[Proposition~13]{Schicho:12d}. The proof, which uses local isomorphisms
of restrictions to different subhabitats, is far from being trivial. The idea to use local isomorphisms
to compare orders on coefficient ideals defined on different hypersurfaces of maximal contact
has been introduced in \cite{Wlodarczyk:07}.

\begin{example}
We consider the singularity $(\langle x^2-y^3\rangle,1)$ on the habitat $(\A^2,())$. We blowup the
origin $(0,0)$, which is contained in the singular locus. 
The blowup variety can be covered by two charts, which already occured in Example~\ref{ex:blow}.
In the first chart, the transform is 
$(\langle (1-x\tilde{y}^3)x\rangle,1)$; in the second chart, it is the ideal-type singularity
$(\langle(\tilde{x}^2-y)y\rangle,1)$. For the transform, the monomial factor is $(1)$;  and the ideal sheaf $\tilde{I}$ is 
$\langle 1-x\tilde{y}^3\rangle$ in the first chart and $\langle\tilde{x}^2-y\rangle$ in the
second chart, hence the maxorder is 1.
\end{example}

\begin{lemma} \label{lem:howto}
Let $A$ be a non-bold gallimaufry of dimension $m$ on a habitat $(W,(E_1,\dots,E_r))$.
Assume, for simplicity, that all hypersurfaces $E_i$ have a non-empty intersection 
with the singular locus.
Let $(a_1,\dots,a_r)$ be the monomial factor of $A$, and let $o$ be the maxorder.

Let $\{i_1,\dots,i_k\}$ be a subset of $\{1,\dots,r\}$.
If $a_{i_1}+\dots+a_{i_k}+o<1$, then $E_{i_1}\cap\dots\cap E_{i_k}\cap\Sing(A)=\emptyset$.
If $a_{i_1}+\dots+a_{i_k}\ge 1$, then $E_{i_1}\cap\dots\cap E_{i_k}\subset\Sing(A)$.

Let $Z\subset\Sing(A)$ be a straight subvariety of $(W,(E_1,\dots,E_r))$, and let $\{i_1,\dots,i_k\}$ be 
the subset of $\{1,\dots,r\}$ of all $i$ such that $Z\subset E_i$. Then the monomial factor
of the transform on the blowup at $Z$ is of the form $(a_1,\dots,a_r,a_{r+1})$, where
\[ a_{i_1}+\dots+a_{i_k}-1 \le a_{r+1} \le a_{i_1}+\dots+a_{i_k}+o-1 . \]
\end{lemma}

\begin{proof}
Straightforward (using local analytic coordinates where the center and all hypersurfaces are 
defined by coordinate functions).
\end{proof}

\begin{example} \label{ex:str}
On the habitat $(\A^3,(x,y,z))$, we consider the singularity \linebreak 
$(\langle x^3y^5z^7(x^2+y^5)\rangle,1)$.
Its monomial factor is $(3,5,7)$, and its maxorder is 2. When we blowup the line $x=y=0$, we get two charts.
In the chart with coordinate functions $(\tilde{x}=\frac{x}{y},y,z)$, the transform is
$(\langle \tilde{x}^3y^{9}z^7(\tilde{x}^2+y^3)\rangle,1)$, and 
in the chart with coordinate functions $(x,\tilde{y}=\frac{y}{x},z)$, the transform is
$(\langle x^{9}\tilde{y}^5z^7(1+x^3\tilde{y}^5)\rangle,1)$.
The monomial factor of the transform is $(3,5,7,9)$.

In this example, the maxorder of the transform is again equal to 2, and the transform of the ideal-type
singularity $(\langle x^2+y^5\rangle,2)$ is $(\langle \tilde{x}^2+y^3\rangle,2)$ in the first chart
and $(\langle 1+x^3\tilde{y}^5\rangle,2)$ in the second chart (see also Lemma~\ref{lem:tt} below).
\end{example}

\begin{definition}
A gallimaufry is called {\em monomial} if its maxorder is 0.
\end{definition}

\begin{lemma}
The transform of a monomial gallimaufry is monomial. If $(a_1$, $\dots$, $a_r)$
is the monomial factor, and if $\{i_1,\dots,i_k\}$ is the subset of hypersurface
indices such that $Z\subset E_i$, then 
$(a_1,\dots,a_r,a_{i_1}+\dots+a_{i_k}-1)$ is the monomial factor of the transform on the blowup at $Z$.
\end{lemma}

\begin{proof}
This is a consequence of Lemma~\ref{lem:howto}.
\end{proof}

\begin{definition}
A non-bold and non-resolved singularity/gallimaufry is {\em tight} 
if it has trivial monomial factor $(0,\dots,0)$ and maxorder~1.

Let $(I,b)$ be an ideal-type singularity such that $I$ is not zero on any component of the habitat.
Let $\left(\frac{e_1}{b},\dots,\frac{e_r}{b}\right)$ be its monomial factor, and let $o$ be the
maxorder. Assume $o>0$. 
The {\em tightification} of a $(I,b)$
is defined as the differential closure of $(\tilde{I},ob)+(I,b)$,
where $\tilde{I}:=\Id(E_1)^{-e_1}\cdots\Id(E_r)^{-e_r}I$
(note that $ob$ is an integer). 

The tightification of a general singularity is defined by passing to an equivalent
ideal type singularity followed by tightification as defined above.

The tightification of a gallimaufry is defined by restriction to a zoom, singularity tightification,
and extension. The tightification of a gallimaufry $A$ is denoted by $\Tf(A)$. If $A$ is bold or resolved,
then $\Tf(A)$ is not defined.
\end{definition}

\begin{example}
Let $m,n$ be positive integers. On the habitat $(\A^1,())$, the ideal-type singularity $(\langle x^m\rangle ,n)$
has maxorder $\frac{m}{n}$, so it is tight if and only if $m=n$. If $m<n$, then the singularity is
resolved. If $m\ge n$, then the tightification is equal to $(\langle x^m\rangle ,m)$.
\end{example}

It is not apparent that the tightification is well-defined for gallimaufries, one has to show
independence of the local choice of the subhabitat of dimension $m$. We refer to \cite[Proof of Theorem~19]{Schicho:12d}
for the proof.
This proof uses local analytic isomorphisms between restrictions to different habitats
(see \cite{Wlodarczyk:07}).

\begin{remark} \label{rem:global}
For given gallimaufry $A$, it would also be possible to compute monomial factor and maxorder using Jacobian
ideals. Hence the computation of the tightification is the only construction of the resolution algorithm 
which uses subhabitats and therefore local choices (the result is independent of the local choices
by the statement before). 
If we had a different construction without local choices we would have a global algorithm for 
resolution of singularities that never passes to local coverings. 
The author does not have an idea for such a construction.
\end{remark}

\begin{lemma} \label{lem:tt}
The transform of a tight gallimaufry is either tight or resolved.

Let $A$ be a non-bold gallimaufry. Let $a=(a_1,\dots,a_r)$ be its monomial factor.
Assume that $A$ has maxorder $o>0$.
Let $Z$ be a straight subvariety in the singular locus of $\Tf(A)$. 
Let $\{i_1,\dots,i_k\}$ be the subset of hypersurface indices such that $Z\subset E_i$. 
Then the monomial factor of $\Tr_Z(A)$ is $a':=(a_1,\dots,a_r,a_{i_1}+\dots+a_{i_k}+o-1)$,
and if $o'$ is the maxorder, then $o'\le o$. Equality holds if and only if $\Tr_Z(\Tf(A))$
is not resolved; and in this case,
$\Tr_Z(\Tf(A))$ is equivalent to $\Tf(\Tr_Z(A))$.
\end{lemma}

\begin{proof}
Locally at any open subset in which the maximal order is assumed, we can restrict to a subhabitat of dimension $m$, 
and it suffices to show the statement for singularities. 
In this situation, the proof is  straightforward (compare with Example~\ref{ex:str}).

In any other open subset, we get a monomial factor $a''$ with last exponent $,a_{i_1}+\dots+a_{i_k}+o_1-1)$
with some $o_1<o$. In the transform of this open subset, the order is bounded by $o_1$.
\end{proof}

\begin{example}
Let $A$ be the singularity $(\langle(x^2-y^n)y^m\rangle,1)$ on the habitat $(\A^2,x)$, where $n\ge 2, m\ge 0$.
Its monomial factor is $(m)$, its maxorder is 2, and its tightification is $(\langle x^2-y^n\rangle,2)$.
Let $Z$ be the point $(0,0)$. The blowup can be covered by two charts, which already occured in Example~\ref{ex:blow}.
In the first, the transform of $A$ is $(\langle(\tilde{x}^2-y^{n-2})y^{m+1}\rangle,1)$. 
In the second, the transform is $(\langle(1-x^{n-2}\tilde{y}^2)x^{m+1}\tilde{y}^m\rangle,1)$.
The monomial factor $(m,m+1)$.
The transform of $\Tf(A)$ is $(\langle\tilde{x}^2-y^{n-2}\rangle,2)$ in the first chart
and $(\langle 1-x^{n-2}\tilde{y}^2\rangle,2)$ in the second chart.
If $n\ge 4$, then the transform of the tightification is the tightification of the transform, and
the maxorder of $\Tr_Z(A)$ is 2.
If $n=2,3$, then the transform of the tightification is resolved, and the maxorder is 1.
The new tightification is $(\langle\tilde{x}^2-y^{n-2}\rangle,1)$ in the first chart.
\end{example}

\begin{lemma}
Let $A$ be a tight gallimaufry of dimension $m>0$ on a habitat $(W,(E_1,\dots,E_r))$. 
Assume that $\Sing(A)\cap E_i=\emptyset$ for $i=1,\dots,r$. Then $A$ descends to dimension $m-1$.
\end{lemma}

\begin{proof}
It suffices to show the statement for singularities.
Moreover, we may assume that $A$ is differentially closed. Then $\Delta(A_1)=\mathcal{O}_W$. 
For any $p\in\Sing(A)$, there exists $f\in (A_1)_p$ of order~1 at $p$. This local section is
defined and still in some open neighbourhood $U$ of $p$. Moreover, the zero set $X$ of $f$ is a hypersurface
in $p$ is a nonsingular point of $X$. We define $U'$ as $U$ minus the singular locus of $X$.
Then $A|_{U'}$ restricts properly to the hypersurface defined by $f$.
\end{proof}

\section{The Game and How to Win It} \label{sec:play}

In this section, we explain the combinatoric part of the resolution.

In any step of the game, the player is given some combinatorial information on a main gallimaufry
which is to be resolved, as well as additional gallimaufries which are related in various ways,
for instance if a gallimaufry is not bold and has positive maxorder then there might be
a tightification. During the game, threads are created; a blowup step adds one gallimaufry to
each existing thread. In the beginning of the game, there is only one thread of length zero with a single
gallimaufry $A$. In the end, this thread should be extended to a resolution of $A$.

This is the combinatorial information on a gallimaufry $A$ on a habitat $(W$, $(E_1$, $\dots$, $E_r))$ 
which is given to the player: 

\begin{itemize}
\item a simplicial complex $\Xi$ with vertices in the set
	$\{1,\dots,r\}$, consisting of all faces $\{i_1,\dots,i_k\}$ such that 
	$E_{i_1}\cap\dots E_{i_k}\cap\Sing(A)\ne\emptyset$;
\item the gallimaufry dimension $m$;
\item a generating degree of $A$;
\item the information whether $A$ is bold or not;
\item if $A$ is not bold, then the monomial factor $a:\mbox{Vertices}(\Xi)\to\Q_{\ge 0}$.
	This is just a labelling og the vertices by rational numbers;
\item the maxorder $o\in\frac{1}{b}\Z$, where $b$ is the generating degree provided as specified above.
\end{itemize}

Note that we have to distinguish the {\em empty set complex} that has no vertices and consists of one -1-complex
whose set of vertices is the empty set, and the {\em empty complex} that has no faces at all.
A gallimaufry is resolved if and only if its complex is the empty complex.

Now the player has to choose a move. There are six possible moves, two blowup moves and four
moves that create additional gallimaufries. We may distinguish two types of blowups.

\begin{description}
\item[Type I] For a gallimaufry $A$ on a habitat $(W,(E_1,\dots,E_r))$ and monomial factor
	$(a_1,\dots,a_r)$ and indices $i_1,\dots,i_k$ such that $a_{i_1}+\dots+a_{i_k}\ge 1$,
	the intersection of $E_1,\dots,E_r$ and some locally defined zoom
	is a straight subvariety contained in the singular locus.
	It is independent of the choice of the zoom because it can also be defined as the intersection of
	$E_1,\dots,E_r$ and the singular locus. A type~I blowup is a blowup of such a subvariety.
	In the winning strategy we describe here, type~I blowups are only needed when $A$ is monomial.
	However, one has to keep in mind that the blowup not only transforms $A$ but also other gallimaufries
	that are given at the same time in the game.
\item[Type II] For a bold gallimaufry $A$ of gallimaufry dimension $m$, the union of all $m$-dimensional
	components of the singular locus is a straight subvariety, by Lemma~\ref{lem:bold}. 
	A type~II blowup is a blowup at such
	a subvariety. By Lemma~\ref{lem:bold}, the transform of $A$ is not bold,
	but again, the blowup also transforms other gallimaufries that are given at this step.
\end{description}

Here is an overview on the possible moves of the player at each turn.

\begin{enumerate}
\item If a gallimaufry complex has a face with sum of labels greater than or equal to 1,
	then she may issue a blowup of type~I on that face.
\item If a gallimaufry is bold, then she may issue a blowup of type~II on that gallimaufry.
\item If a gallimaufry is tight and its complex $\Xi$ is the empty set complex,
	then she may issue a descent.
\item If a gallimaufry is not bold and has maxorder $o>0$, 
	then she may issue a tightification.
\item For some gallimaufry with complex $\Xi$ and vertex $j\in\mbox{Vertices}(\Xi)$, 
	she may issue a relaxation. This will create a gallimaufry with a smaller 
	sequence of hypersurfaces, as defined below.
\item For some gallimaufry with complex $\Xi$ and vertex $j\in\mbox{Vertices}(\Xi)$, 
	she may issue an intersection.
\end{enumerate}

A blowup move includes a unique specification of the blowup center: for type~I, the center
is the intersection of the singular locus with all hypersurfaces in the habitat sequence
corresponding to the vertices of the face occuring in the description of the move;
for type~II, it is the $m$-dimensional locus of the singular locus, where $m$ is the
gallimaufry dimension.
The consequences of a blowup move are that the habitat is blown up at the indicated center $Z$ and
all gallimaufries with $Z\subset\Sing(A)$ are transformed, so that their threads are prolongued. 
The remaining threads are removed, their threads are differentially closed.
The combinatorial data of the transformed gallimaufries are partially determined by the
properties of gallimaufries in the previous sections. 
The dimension
and generating degree are never changed. Also, the relation between two gallimaufries in two threads
(descent, quotient, intersection) are kept. The remaining data (for instance maxorder of non-tight gallimaufries)
are again given to the player.

In the remaining moves, a new gallimaufry is created and a thread is opened starting with it.
In a descent move, the new gallimaufry is given by the same Rees algebra on the same habitat, but it
is considered as a gallimaufry in dimension one less; this has an effect on the monomial factor, on
the maxorder, and on the boldness property. 
In a relaxation move with gallimaufry $A$ and vertex $j$, the new gallimaufry is defined by the same
Rees algebra $A$, but the habitat is changed: $E_j$ is replaced by $\emptyset$. 
In an intersection move with gallimaufry $A$ and vertex $j$, the new gallimaufry is $A+(\Id(E_j),1)$.

\begin{remark}
It is possible that an intersection move follows a relaxation move for the same vertex. In this situation,
we do not want to form the sum with $(\Id(\emptyset),1)$ because this is the unit singularity.
So by convention, the added summand in the intersection move is always computed from the habitat
in the main thread.
\end{remark}

\begin{lemma} \label{lem:hiro}
There is a winning strategy for monomial gallimaufries.
\end{lemma}

\begin{proof}
The winning strategy is to blowup a minimal face among all faces with sum of labels greater than
or equal to 1. Then the complex of the
transformed gallimaufry is a subdivision, where the label sum of any of the new faces is strictly smaller 
than the label sum of the old face which contains that new face
topologically and which has disappeared in the subdivision. This is only possible a finite number of
times because the labels are in $\frac{1}{b}\Z_{\ge 0}$, where $b$ is the generating degree.
\end{proof}

\begin{lemma} \label{lem:redorder}
Let $m\ge 0$ be an integer. If there is a winning strategy for tight gallimaufries of dimension $m$,
then there is a winning strategy for all gallimaufries of dimension $m$.
\end{lemma}

\begin{proof}
By a type~II blowup, we may reduce to a non-bold gallimaufry $A$. Then we have a maxorder 
$o\in\frac{1}{b}\Z_{\ge 0}$, where $b$ is the generating degree. If $o>0$, then we create the tightification
$\Tf(A)$. By assumption, there is a resolution of $\Tf(A)$. The sequence of blowups defines a thread starting
with $A$. By Lemma~\ref{lem:tt}, the last singularity $A'$ of this thread has maxorder $o'<o$.
If $o'>0$, then we start a new thread starting with $\Tf(A')$ and apply the winning strategy to the tight
gallimaufry. The resolution of the $\Tf(A')$ induces a prolongation of the thread of $A$ passing $A'$ and
ending with a singularity $A''$ of maxorder $o''<o'$. Since the maxorder can only drop finitely many times, 
we eventually achieve the monomial case $o=0$, which can be won by Lemma~\ref{lem:hiro}.
\end{proof}

\begin{lemma} \label{lem:induction}
Let $m>0$ be an integer. If there is a winning strategy for gallimaufries of
dimension $m-1$, then there is a winning strategy for tight gallimaufries of dimension $m$.
\end{lemma}

\begin{proof}
Let $A$ be a gallimaufry of dimension $m$, and let $\Xi$ be its simplicial complex. If $\Xi$ is
the empty set complex, then $A$ is tight, so it descends to dimension $m-1$. Then we can construct a resolution
by assumption (the resolution for the descent is also a resolution for $A$ itself).

In general, we create a relaxation gallimaufry $B$ on $(W,())$. Since $B$ is tight and its complex
is the empty set complex, $B$ descends to dimension $m-1$. In the following, we construct a ``careful''
resolution of $B$. The extra care is necessary to avoid blowing up centers that are not straight
for the habitat of $A$. 

Let $f$ be a maximal face of $\Xi$. We form the intersection gallimaufry $B_f$ with respect to all
vertices of $f$ and construct a resolution for $B_f$. The blowup centers are contained in the
intersection of the hypersurfaces corresponding to $f$, therefore they are also straight for the habitat
of the singularity in the thread of $A$. Then we set $A'$ and $B'$ to be the last singularities of
the threads of $A$ and $B$, and $\Xi'$ to be the complex obtained by removing the face $f$ from $\Xi$.
Again, the vertices of $\Xi'$ are the indices of hypersurfaces that have been relaxed 
in the thread of $B$, and the faces correspond to sets of hypersurfaces which have a non-empty intersection
within the singular locus in the singularity in the thread of $A$.

If $\Xi'$ is not the empty complex, then we choose a maximal face $f'$ of $\Xi'$ and repeat.
In each step, the number of faces of $\Xi$ drops.
After finitely many steps, we get the empty complex, and $A$ is resolved.
\end{proof}

\begin{remark}
There are no tight gallimaufries of dimension $m=0$. Actually, a gallimaufry of dimension~0
is either bold or resolved, hence it can be resolved in at most one step.
\end{remark}

\begin{theorem} \label{thm:win}
Every gallimaufry has a resolution.
\end{theorem}

\begin{proof}
This is now an obvious consequence of the three lemmas and the remark above.
\end{proof}

As a consequence of Theorem~\ref{thm:win} and Theorem~\ref{thm:reduce}, every irreducible variety
over a field of characteristic zero has a resolution.

\begin{example}
Let $A$ be the differential closure of $(\langle xy(x-y)\rangle,2)$ on the habitat $(\A^2,())$.
In order to resolve the singularity it, one would
start the game by giving the singularity to Mephisto. He would then give the following information
to Dido: {\em In thread $T_0$ (the main thread), we have dimension $2$ and generating degree $2$.
Currently, its complex is the empty set complex, the gallimaufry is not bold, the monomial
factor is obvious, and the maxorder is 1.} 

Since $T_0$ is tight, Dido will now descend $T_0$, creating a thread $T_1$ of dimension $1$
(using Lemma~\ref{lem:induction} for winning the game). Mephisto would then tell Dido that
the maxorder of $T_1$ is again 1.

Since $T_1$ is tight, Dido will descend $T_1$, creating a thread $T_2$ of dimension $0$. 
Mephisto will then tell Dido that $T_2$ is bold.

Now, Dido will issue a type~II blowup on $T_2$. The blowup at the $0$-dimensional part of the singular
locus (which in this case coincides with the whole singular locus) will resolve all threads.
The game is won and the singularity in thread $T_0$ is resolved.
\end{example}

The game described in this section was played in the Clay Summer School in Obergurgl. The participants
formed two teams, called Mephisto and Dido. Each team was working on a blackboard that was not readable
by the other team. The actual player was Dido, and it was Mephisto's task to provide the combinatorial
information by computation. The singularity was not revealed to Dido, but the team did guess it.

\begin{enumerate}
\item Mephisto was given the habitat $(\A^2,())$ and the singularity $(\langle x^2-y^3\rangle,1)$.
	After a short computation, Mephisto gave Dido a piece of paper with the following information:
	{\em In thread $T_0$ (the main thread), you have dimension $2$ and generating degree $1$.
	Currently, its complex is the empty set complex, the gallimaufry is not bold, the monomial
	factor is obvious, and the maxorder is 2.}
\item Dido decided to tightify $T_0$, creating the thread $T_1$. The dimension of $T_1$
	is 2, the monomial factor is trivial and the maxorder is 1; this is already clear.
\item Mephisto computed the tightification of $T_0$: it is $(\langle x^2-y^3\rangle,2)$. 
	The information given to Dido was: {\em $T_1$ has generating degree $2$. 
	Its complex is the empty set complex} (this could have been deduced by Dido, because the
	complex of the tightification is always a non-empty subcomplex).
\item Dido decided to descend $T_1$, creating the thread $T_2$. Its dimension is 1. The generating degree
	and the complex is inherited from $T_1$. 
\item To compute the maxorder, Mephisto restricted $T_2$ to the subvariety defined by $x$. The restriction
	is $(\langle y^3\rangle,2)$. Mephisto told Dido: {\em $T_2$ is not bold, has trivial monomial factor, 
	and maxorder $\frac{3}{2}$.}
\item Dido decided to tightify $T_2$, creating the thread $T_3$.
\item Mephisto computed the tightification $(\langle y\rangle,1)$ and told Dido: {\em the generating degree of $T_3$ 
	is 1, and the complex is the empty set complex.}
\item Dido decided to descend $T_3$, creating the thread $T_4$ of dimension~0.
\item Mephisto told Dido: {\em $T_4$ is bold.}
\item Dido demanded a type~II blowup on $T_4$.
\item Mephisto computed the blowup and transforms. In the first chart (which is the interesting one),
	the habitat is $(\A^2,(y))$, and the main singularity in the thread $T_0$ is
	$(\langle (\tilde{x}^2-y)y\rangle,1)$. The other singularities are resolved.
	Mephisto told Dido: {\em The complex in $T_0$ consist of the 0-face $\{\tt 1\}$ and the empty face. 
	The threads $T_1,T_2,T_3,T_4$ are resolved. The monomial factor of $T_0$ is $(1)$, and the
	maxorder is 1.}
\item Dido decided to tightify $T_0$, creating a thread $T_5$.
\item Mephisto computed the tightification $(\langle\tilde{x}^2-y\rangle,1)$ and told Dido: 
	{\em The complex of $T_5$ is currently the full complex of $T_0$.
	The generating degree is 2000. This is correct, it is not required to give
	the minimal generating degree to the player.}
\item Dido decided to intersect $T_5$ with vertex $\tt 1$, creating thread $T_6$.
	For an intersection move, the combinatorial data
	can be inferred, so Mephisto does not need to provide information.
\item Dido decided to relax vertex $\tt 1$ from $T_6$, creating $T_7$. Also here, no additional information
	from Mephisto is needed.
\item Dido decided to descend $T_7$, creating the thread $T_8$ of dimension~1.
\item The restricted singularity is $(\langle \tilde{x}^2\rangle,1)$.
	Mephisto told Dido: {\em Currently, the monomial factor of $T_8$ is trivial and the maxorder is 2.}
\item Dido decided to tightify $T_8$, creating a thread $T_9$.
\item For $T_8$, Mephisto gave the generating degree 2000. The complex of $T_8$ is currently the empty set
	complex (because $\tt 1$ was relaxed).
\item Dido decided to tightify $T_9$, creating the thread $T_{10}$ of dimension 0.
\item Mephisto announced that $T_{10}$ is bold.
\item Dido demanded a type~II blowup at $T_{10}$.
\end{enumerate}

The game went on for some time, but after 90 minutes the game was interrupted without a victory of Dido.
No doubt Dido would have won when given more time, because the members of the team already had a clear
strategy.

\bibliography{alles}
\bibliographystyle{plain}

\end{document}